\documentclass[a4paper,12pt]{article}
\usepackage[top=2.5cm,bottom=2.5cm,left=2.5cm,right=2.5cm]{geometry}
\usepackage{cite, amsmath, amssymb}
\usepackage[margin=1cm,%
                font=small,%
                format=hang,%
                labelsep=period,%
                labelfont=bf]{caption}
\usepackage{amsthm}    
\usepackage{amsfonts,amscd} 
\usepackage[all]{xy}

\makeatletter
\newfont{\footsc}{cmcsc10 at 8truept}
\newfont{\footbf}{cmbx10 at 8truept}  
\newfont{\footrm}{cmr10 at 10truept} 
\makeatother

\newtheorem{theorem}{Theorem}

\newtheorem{Lem}{Lemma}

\title{Asymmetric Circular Graph with Hosoya Index and Negative Continued Fractions} 

\author{Takao Komatsu
\\
\small Department of Mathematical Sciences, School of Science\\
\small Zhejiang Sci-Tech University\\
\small Hangzhou 310018, China\\
\small \texttt{komatsu@zstu.edu.cn}
}

\date{
}

\allowdisplaybreaks[2] 

\begin{document}
\maketitle


\bigskip

\begin{abstract}  
It has been known that the Hosoya index of caterpillar graph can be calculated as the numerator of the simple continued fraction. Recently, the author \cite{Komatsu2020} introduces a more general graph called caterpillar-bond graph and shows that its Hosoya index can be calculated as the numerator of the general continued fraction.  

In this paper, we show how the Hosoya index of the graph with non-uniform ring structure can be calculated from the negative continued fraction. We also give the relation between some radial graphs and multidimensional continued fractions in the sense of the Hosoya index.        
 \\
\end{abstract}

\baselineskip=0.30in

\section{Introduction} 

The concept of the {\it topological index} was first introduced by Haruo Hosoya in 1971 \cite{Hosoya1971} as the integer $Z:=Z(G)$ is the sum of a set of the numbers $p(G,k)$, which is the number of ways for choosing $k$ disjoin edges from $G$. By using the set of $p(G,k)$, the topological index $Z$ is defined by 
$$
Z=\sum_{k=0}^m p(G,k)\,. 
$$   
As more different types of topological indices have been discovered in chemical graph theory (e.g., see \cite{DB}), the first topological index is also called {\it Hosoya index} or the {\it $Z$ index} nowadays.  Topological indices are used for example in the development of quantitative structure-activity relationships (QSARs) in which the biological activity or other properties of molecules are correlated with their chemical structure.  

The topological index is closely related to Fibonacci $F_n$ \cite{Hosoya1973} and related numbers \cite{Hosoya2007c}.  For the path graph $S_n$, we have $Z(S_n)=F_{n+1}$, where $F_n=F_{n-1}+F_{n-2}$ ($n\ge 2$) with $F_0=0$ and $F_1=1$.  For the monocyclic graph $C_n$, we have $Z(C_n)=L_n$, where $L_n$ is the Lucas number, defined by $L_n=L_{n-1}+L_{n-2}$ ($n\ge 2$) with $L_0=2$ and $L_1=1$.  
 
In \cite{Hosoya2007}, manipulation of continued fraction, either finite and infinite, was shown to be greatly simplified and systematized by introducing the topological index $Z$ and caterpillar graph $C_n(x_1,x_2,\dots,x_n)$. 
A {\it caterpillar graph} is a tree containing a path graph such that every edge has one or more endpoints in that path.   In \cite{Hosoya2007}, it is shown that for $n\ge 1$ 
\begin{equation}
Z\bigl(C_n(a_0,a_1,\dots,a_{n-1})\bigr)=p_{n-1}\,, 
\label{caterpillar-cf} 
\end{equation}
where 
$$
\frac{p_{n-1}}{q_{n-1}}=a_0+\cfrac{1}{a_1+{\atop\ddots+\cfrac{1}{a_{n-1}}}}\quad\hbox{with}\quad \gcd(p_{n-1},q_{n-1})=1,~ a_i\ge 1~(0\le i\le n-1)\,. 
$$ 
In \cite{LZMC}, it is shown that the Hosoya index of the fractal graph can be given by the order and degrees of all the vertices of the graph.   

In \cite{Komatsu2020}, we give graphs whose topological index are exactly equal to the number $u_n$, satisfying the three-term recurrence relation 
$$
u_n=a u_{n-1}+b u_{n-2}\quad(n\ge 2)\quad u_0=0\quad\hbox{and}\quad u_1=u\,,   
$$  
where $a$, $b$ and $u$ are positive integers. This is illustrated by the so-called {\it caterpillar-bond graph} $D_n(x_1,x_2,\dots,x_n;y_1,\dots,y_{n-1})$. We show an interpretation from the continued fraction expansion in a more general case, so that the topological index can be computed easily. We also show how to calculate Hosoya index of the given tree graph or the graph including circle type graphs, by using the branched continued fractions. 
The Hosoya index can be calculated easily for the graph with uniform ring structure by transforming it into the normal caterpillar-bond graph. 

In this paper, we show how the Hosoya index of the graph with non-uniform ring structure can be calculated from the negative continued fraction. We also give the relation between some radial graphs and multidimensional continued fractions in the sense of the Hosoya index.

\section{Preliminaries}  

Any real number can be expressed as a generalized continued fraction expansion of the form 
$$ 
\alpha
=a_0+\cfrac{b_1}{a_1+\cfrac{b_2}{a_2+{\atop\ddots}}}\,.  
$$  
In this paper, we assume that all numbers $a_0,a_1,a_2,\dots$ and $b_1,b_2,\dots$ are positive integers.  
The $n$-th convergent $p_n/q_n$ is given by 
\begin{align*}   
\frac{p_n}{q_n}
&=a_0+\cfrac{b_1}{a_1+\cfrac{b_2}{a_2+{\atop\ddots{\atop\displaystyle +\dfrac{b_n}{a_n}}}}}\\
&:=a_0+\frac{b_1}{a_1}{\atop+}\frac{b_2}{a_2}{\atop+\cdots+}\frac{b_n}{a_n}\,. 
\end{align*} 
Here, $p_n$ and $q_n$ satisfy the recurrence relation: 
\begin{align}
p_n&=a_n p_{n-1}+b_n p_{n-2}\quad(n\ge 2),& p_{0}&=a_0, &p_{1}&=a_0 a_1+b_1,\label{rec:pn}\\
q_n&=a_n q_{n-1}+b_n q_{n-2}\quad(n\ge 2),& q_{0}&=1, &q_{1}&=a_1\,.
\label{rec:qn}
\end{align}

Notice that the expression of the generalized continued fraction expansion is not unique, and $p_n$ and $q_n$ are not necessarily coprime.  

In \cite{Komatsu2020}, we introduced a combined graph of the caterpillar graph and the bond graph as their generalization. That is,  
{\it caterpillar-bond graph}  $D_n(x_1,x_2,\dots,x_n;y_1,y_2,\dots,y_{n-1})$

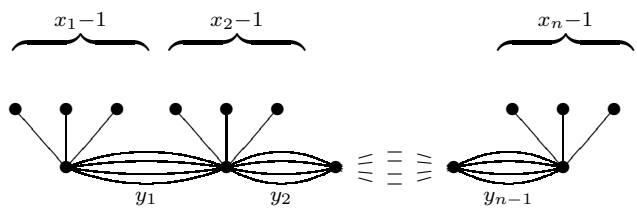
\begin{align*} 
&\overbrace{\phantom{\hspace{0.7in}}}^{x_1-1}~ 
\overbrace{\phantom{\hspace{0.7in}}}^{x_2-1}\qquad\qquad\quad\quad  
\overbrace{\phantom{\hspace{0.7in}}}^{x_n-1}\\* 
&\xymatrix@=16pt{ 
*=0{\bullet}&*=0{\bullet}&*=0{\bullet}&*=0{\bullet}&*=0{\bullet}&*=0{\bullet}&&&&*=0{\bullet}&*=0{\bullet}&*=0{\bullet}\\
&*=0{\bullet}\ar @{-} [lu]\ar @{-} [u]\ar @{-} [ru]\ar @/^/ @{-} [rrr]\ar @/^0.2pc/ @{-} [rrr] \ar @/_0.2pc/ @{-} [rrr]  
\ar @/_/ @{-} [rrr]_{y_1} &&&*=0{\bullet}\ar @{-} [lu]\ar @{-} [u]\ar @{-} [ru]\ar @/^/ @{-} [rr]\ar @/^0.2pc/ @{-} [rr] \ar @/_0.2pc/ @{-} [rr] 
\ar @/_/ @{-} [rr]_{y_2}&&*=0{\bullet}\ar @/^/ @{--} [rr]\ar @/^0.2pc/ @{--} [rr] \ar @/_0.2pc/ @{--} [rr]  
\ar @/_/ @{--} [rr]&&*=0{\bullet}&&*=0{\bullet}\ar @{-} [lu]\ar @{-} [u]\ar @{-} [ru]\ar @/^/ @{-} [ll]^{y_{n-1}}\ar @/^0.2pc/ @{-} [ll] \ar @/_0.2pc/ @{-} [ll]\ar @/_/ @{-} [ll]&\\
}
\end{align*}

The Hosoya index of the caterpillar-bond graph can be calculated easily by using the continued fraction expansion \cite{Komatsu2020}.    

\begin{Lem}  
For $n\ge 1$, 
$$
Z\bigl(D_n(a_0,a_1,\dots,a_{n-1};b_1,\dots,b_{n-1})\bigr)=p_{n-1}\,, 
$$ 
where $p_{n-1}$ is the numerator of the convergent of the continued fraction expansion 
\begin{equation}   
\frac{p_{n-1}}{q_{n-1}}
=a_0+\cfrac{b_1}{a_1+\cfrac{b_2}{a_2+{\atop\ddots{\atop\displaystyle +\dfrac{b_{n-1}}{a_{n-1}}}}}} 
\label{pnqn}  
\end{equation}   
and $p_j$'s and $q_j$'s satisfy the recurrence relations in (\ref{rec:pn}) and (\ref{rec:qn}), respectively. 
\label{lem:main}  
\end{Lem}

The other graphs, including monocycle graph, cannot be calculated by this formula directly. But, by the equivalent transformation, it is possible to calculate the Hosoya index for some circular graphs with perfect point symmetry structure \cite{Komatsu2020}. For example, the cycle graph C$_n$ can be transformed into the caterpillar-bond graph 
$${\rm D}_{n-1}(\underbrace{1,\dots,1}_{n-2},2;2,\underbrace{1,\dots,1}_{n-3})\,. 
$$   
It is so understandable that ${\rm Z}({\rm C}_n)=L_n$ (\cite{Hosoya2007c}), where $L_n$ are Lucas numbers, defined by $L_n=L_{n-1}+L_{n-2}$ ($n\ge 2$) with $L_0=2$ and $L_1=1$.  
For example, we can see it by cycloraraffin C$_n$H$_{2 n}$.   

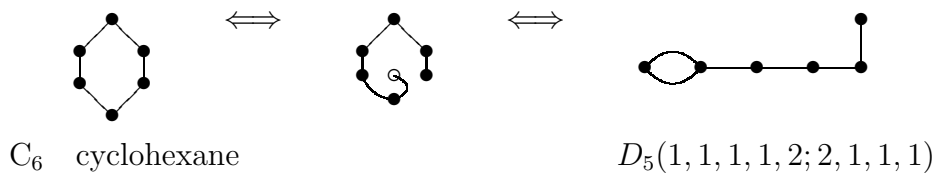
\begin{align*}   
&\qquad\xymatrix@=6pt{
&*=0{\bullet}\ar@{-}[ld]\ar@{-}[rd]&\\
*=0{\bullet}\ar@{-}[d]&&*=0{\bullet}\ar@{-}[d]\\
*=0{\bullet}\ar@{-}[rd]&&*=0{\bullet}\ar@{-}[ld]\\
&*=0{\bullet}&
}\qquad\Longleftrightarrow\qquad 
\xymatrix@=6pt{
&*=0{\bullet}\ar@{-}[ld]\ar@{-}[rd]&\\
*=0{\bullet}\ar@{-}[d]&&*=0{\bullet}\ar@{-}[d]\\
*=0{\bullet}\ar@/_0.2pc/@{-}[rd]&*=0{\circ}&*=0{\bullet}\\
&*=0{\bullet}\ar@/_0.4pc/@{-}[u]&
}\qquad\Longleftrightarrow\qquad 
\xymatrix@=15pt{  
&&&&*=0{\bullet}\ar@{-}[d]\\
*=0{\bullet}\ar@/^/@{-}[r]\ar@/_/@{-}[r]&*=0{\bullet}\ar@{-}[r]&*=0{\bullet}\ar@{-}[r]&*=0{\bullet}\ar@{-}[r]&*=0{\bullet}
}\\*
& {\rm C}_6\quad \hbox{cyclohexane}\qquad\qquad\qquad\qquad\qquad\qquad {D}_5(1,1,1,1,2;2,1,1,1)
\end{align*} 

In addition, comb related graphs, including monocycle graphs C$_n$ and cyclic comb graphs CU$_n$ \cite{Hosoya2007c}, can be converted into the caterpillar-bond graphs.

Indeed, 
CV$_n$ can be converted into D$_n(\underbrace{3,\dots,3}_n;2,\underbrace{1,\dots,1}_{n-2})$ by cutting one edge to another edge into bond edges by turning around.    
$$  
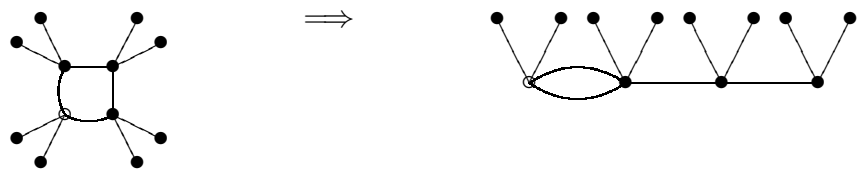
\xymatrix@=3pt{
&*=0{\bullet}\ar@{-}[rdd]&&&&*=0{\bullet}\ar@{-}[ldd]&\\
*=0{\bullet}\ar@{-}[rrd]&&&&&&*=0{\bullet}\ar@{-}[lld]\\
&&*=0{\bullet}\ar@{-}[rr]\ar@/_0.2pc/@{-}[dd]&&*=0{\bullet}\ar@{-}[dd]&&\\
&&&&&&\\
&&*=0{\circ}\ar@/_0.2pc/@{-}[rr]&&*=0{\bullet}&&\\
*=0{\bullet}\ar@{-}[rru]&&&&&&*=0{\bullet}\ar@{-}[llu]\\
&*=0{\bullet}\ar@{-}[ruu]&&&&*=0{\bullet}\ar@{-}[luu]&
}
\qquad\qquad  \Longrightarrow \qquad\qquad
\xymatrix@=6pt{  
*=0{\bullet}\ar@{-}[rdd]&&*=0{\bullet}\ar@{-}[ldd]&*=0{\bullet}\ar@{-}[rdd]&&*=0{\bullet}\ar@{-}[ldd]&*=0{\bullet}\ar@{-}[rdd]&&*=0{\bullet}\ar@{-}[ldd]&*=0{\bullet}\ar@{-}[rdd]&&*=0{\bullet}\ar@{-}[ldd]\\
&&&&&&&&&&&\\
&*=0{\circ}\ar@/^/@{-}[rrr]\ar@/_/@{-}[rrr]&&&*=0{\bullet}\ar@{-}[rrr]&&&*=0{\bullet}\ar@{-}[rrr]&&&*=0{\bullet}&
} 
$$ 

In general, let ${\rm C}_{n,a,b}$ be the graph with $a$ additional branches at each vertex and a $b$-tuple on each edge of monocycle ${\rm C}_n$.     

Then, ${\rm C}_{n,a,b}$ can be transformed into the caterpillar-bond graph 
$${\rm D}_n(\underbrace{a+1,\dots,a+1}_n;2 b,\underbrace{b,\dots,b}_{n-2})$$ without changing of the numbers of vertices and edges. That is, we have \cite{Komatsu2020} 
$$
{\rm Z}({\rm C}_{n,a,b})={\rm Z}\bigl({\rm D}_n(\underbrace{a+1,\dots,a+1}_n;2 b,\underbrace{b,\dots,b}_{n-2})\bigr)\,. 
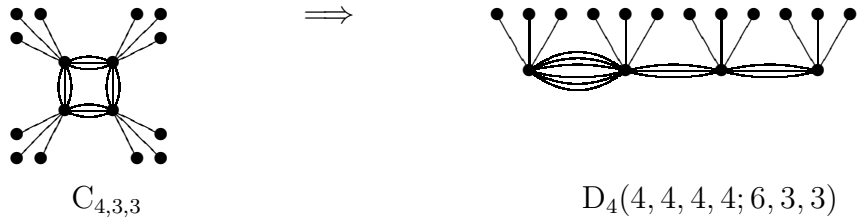
$$ 

\begin{align*}  
&\xymatrix@=3pt{
*=0{\bullet}&*=0{\bullet}\ar@{-}[rdd]&&&&*=0{\bullet}\ar@{-}[ldd]&*=0{\bullet}\\
*=0{\bullet}\ar@{-}[rrd]&&&&&&*=0{\bullet}\ar@{-}[lld]\\
&&*=0{\bullet}\ar@{-}[rr]\ar@/^0.2pc/@{-}[rr]\ar@/_0.2pc/@{-}[rr]\ar@{-}[dd]\ar@/^0.2pc/@{-}[dd]\ar@/_0.2pc/@{-}[dd]\ar@{-}[lluu]&&*=0{\bullet}\ar@{-}[dd]\ar@/^0.2pc/@{-}[dd]\ar@/_0.2pc/@{-}[dd]\ar@{-}[rruu]&&\\
&&&&&&\\
&&*=0{\bullet}\ar@{-}[rr]\ar@/^0.2pc/@{-}[rr]\ar@/_0.2pc/@{-}[rr]\ar@{-}[lldd]&&*=0{\bullet}\ar@{-}[rrdd]&&\\
*=0{\bullet}\ar@{-}[rru]&&&&&&*=0{\bullet}\ar@{-}[llu]\\
*=0{\bullet}&*=0{\bullet}\ar@{-}[ruu]&&&&*=0{\bullet}\ar@{-}[luu]&*=0{\bullet}
}\qquad\qquad \Longrightarrow \qquad\qquad 
\xymatrix@=6pt{  
*=0{\bullet}\ar@{-}[rdd]&*=0{\bullet}\ar@{-}[dd]&*=0{\bullet}\ar@{-}[ldd]&*=0{\bullet}\ar@{-}[rdd]&*=0{\bullet}\ar@{-}[dd]&*=0{\bullet}\ar@{-}[ldd]&*=0{\bullet}\ar@{-}[rdd]&*=0{\bullet}\ar@{-}[dd]&*=0{\bullet}\ar@{-}[ldd]&*=0{\bullet}\ar@{-}[rdd]&*=0{\bullet}\ar@{-}[dd]&*=0{\bullet}\ar@{-}[ldd]\\
&&&&&&&&&&&\\
&*=0{\bullet}\ar@/^0.2pc/@{-}[rrr]\ar@/_0.2pc/@{-}[rrr]\ar@/^0.4pc/@{-}[rrr]\ar@/_0.4pc/@{-}[rrr]\ar@/^0.6pc/@{-}[rrr]\ar@/_0.6pc/@{-}[rrr]&&&*=0{\bullet}\ar@{-}[rrr]\ar@/^0.2pc/@{-}[rrr]\ar@/_0.2pc/@{-}[rrr]&&&*=0{\bullet}\ar@{-}[rrr]\ar@/^0.2pc/@{-}[rrr]\ar@/_0.2pc/@{-}[rrr]&&&*=0{\bullet}&
}\\*
&\qquad{\rm C}_{4,3,3}\qquad\qquad\qquad\qquad\qquad\qquad\qquad {\rm D}_4(4,4,4,4;6,3,3) 
\end{align*}


However, there was no way to transform a circle type graph without uniform pattern into a caterpillar-bond graph directly. 
For example, consider the very famous benzene C$_6$H$_6$.  

$$
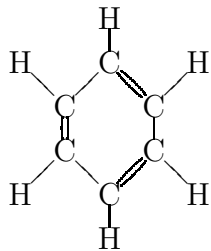
\xymatrix@=8pt{ 
&&*{\rm H}\ar@{-}[d]&&\\
*{\rm H}\ar@{-}[rd]&&*{\rm C}\ar@{-}[ld]\ar@/^0.1pc/@{-}[rd]\ar@/_0.1pc/@{-}[rd]&&*{\rm H}\ar@{-}[ld]\\
&*{\rm C}\ar@/^0.1pc/@{-}[d]\ar@/_0.1pc/@{-}[d]&&*{\rm C}\ar@{-}[d]&\\
&*{\rm C}\ar@{-}[ld]\ar@{-}[rd]&&*{\rm C}\ar@{-}[rd]\ar@/^0.1pc/@{-}[ld]\ar@/_0.1pc/@{-}[ld]&\\
*{\rm H}&&*{\rm C}\ar@{-}[d]&&*{\rm H}\\
&&*{\rm H}&&
}
$$   
In other words, there has been no convenient method to calculate Hosoya index of a circle type graph without uniform pattern directly by using continued fractions.   
In the next section, we show a relation between such graphs and negative continued fractions.

\section{Graph with non-uniform ring structure and negative continued fractions}  

In order to prove our main result, we need the known relations, which were first suggested by Hosoya \cite{Hosoya1971,Hosoya1973} and were elaborated by Gutman and Polansky \cite{GP}.  Though we need only the first one in this paper, we also list related relations for convenience.   

\begin{Lem}  
\begin{enumerate}  
\item If $e=u v$ is an edge of a graph $G$, then $Z(G)=Z(G-e)+Z(G-\{u,v\})$.  
\item If $v$ is a vertex of a graph $G$, then $Z(G)=Z(G-v)+\sum_{u v}Z(G-u v)$, where the summation extends over all vertices adjacent to $v$.   
\item If $G_1,G_2,\dots,G_k$ are connected components of $G$, then $Z(G)=\prod_{i=1}^k Z(G_i)$.  
\end{enumerate} 
\label{gp} 
\end{Lem}  

First, we consider a graph in which the annular part has a staggered structure of $r$ single and $s$ double lines, with $m$ branches at each vertex. We denote it by ${\rm U}_{n,m}^{(r,s)}$, where $n,r,s\ge 1$ and $m\ge 0$ are integers. For simplicity, its Hosoya index is denoted by $u_n:={\text Z}({\rm U}_{n,m}^{(r,s)})$. Hosoya index of ${\rm U}_{n,m}^{(r,s)}$ is given by the numerator of the convergent of a {\it negative} (or backward) {\it continued fraction}.  

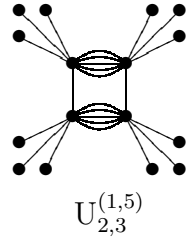
\begin{align*}  
&\xymatrix@=4pt{
*=0{\bullet}&*=0{\bullet}\ar@{-}[rdd]&&&&*=0{\bullet}\ar@{-}[ldd]&*=0{\bullet}\\
*=0{\bullet}\ar@{-}[rrd]&&&&&&*=0{\bullet}\ar@{-}[lld]\\
&&*=0{\bullet}\ar@{-}[rr]\ar@/^0.2pc/@{-}[rr]\ar@/_0.2pc/@{-}[rr]\ar@/^0.4pc/@{-}[rr]\ar@/_0.4pc/@{-}[rr]\ar@{-}[dd]\ar@{-}[lluu]&&*=0{\bullet}\ar@{-}[dd]\ar@{-}[rruu]&&\\
&&&&&&\\
&&*=0{\bullet}\ar@{-}[rr]\ar@/^0.2pc/@{-}[rr]\ar@/_0.2pc/@{-}[rr]\ar@/^0.4pc/@{-}[rr]\ar@/_0.4pc/@{-}[rr]\ar@{-}[lldd]&&*=0{\bullet}\ar@{-}[rrdd]&&\\
*=0{\bullet}\ar@{-}[rru]&&&&&&*=0{\bullet}\ar@{-}[llu]\\
*=0{\bullet}&*=0{\bullet}\ar@{-}[ruu]&&&&*=0{\bullet}\ar@{-}[luu]&*=0{\bullet}
}\\*
&\qquad{\rm U}_{2,3}^{(1,5)}
\end{align*}

\begin{theorem}  
For integers $n,r,s\ge 1$ and $m\ge 0$, we have 
$$
\frac{u_{n}}{u_{n-1}}=\underbrace{M-\frac{r s}{M}{\atop-}\frac{r s}{M}{\atop-\dots-}\frac{r s}{M}}_{n}{\atop-}\frac{M}{2}\,, 
$$ 
where $M=D_2(m+1,m+1;r+s)=(m+1)^2+r+s$.  
Furthermore, for $n\ge 1$ we have 
$$
u_{n}=\left(\frac{M+\sqrt{M^2-4 r s}}{2}\right)^n+\left(\frac{M-\sqrt{M^2-4 r s}}{2}\right)^n\,.
$$ 
\label{th:non-uniform}
\end{theorem}  

\noindent 
{\it Remark.}  
(1) 
In order to keep the recurrence relations, cancelations are not done in the calculation from the continued fraction to the convergent. Namely, one should calculate as 
$$
\frac{a}{b}+\frac{c}{d}=\frac{a d+b c}{b d}
$$  
even though $\gcd(b,d)\ne 1$.  For example, 
$$
\frac{2}{3}+\frac{3}{3}=\frac{15}{9}\quad\hbox{and}\quad \frac{3}{4}+\frac{5}{6}=\frac{38}{24}
$$ 
instead of $5/3$ and $19/12$, respectively.  

\noindent 
(2) If $M>2 r s$, the negative continued fraction in Theorem \ref{th:non-uniform} can be converted (see, e.g.,  \cite{vdP}) into 
$$
\frac{u_{n}}{u_{n-1}}=M-1+\underbrace{\frac{1}{1}{\atop+}\frac{r s}{M-r s-1}{\atop+}\frac{1}{1}{\atop+}\frac{r s}{M-r s-1}{\atop+\dots+}\frac{r s}{M-r s-1}{\atop+}\frac{1}{1}}_{2 n-3}{\atop+}\frac{2 r s}{M-2 r s}\,, 
$$ 
which corresponds with the caterpillar-bond graph $D_{2 n-1}(M-1,1,M-r s-1,1,M-r s-1,\dots,M-r s-1,1,M-2 r s;1,r s,1,r s,\dots,r s,1,2 r s)$. 

\noindent 
{\bf Example.}  
The structure of benzene C$_6$H$_6$ is represented by ${\rm U}_{3,1}^{(2,1)}$. Hence, we have 
$$
\frac{u_3}{u_2}=7-\cfrac{2}{7-\cfrac{2}{7-\cfrac{7}{2}}}=\frac{301}{45}\,. 
$$ 
Therefore, its Hosoya index is given by ${\text Z}({\text U}_{3,1}^{(2,1)})=301$.  In fact,  
$$
{\text Z}({\rm U}_{3,1}^{(2,1)})=\left(\frac{7+\sqrt{41}}{2}\right)^3+\left(\frac{7-\sqrt{41}}{2}\right)^3=301\,. 
$$ 
When $m=0$, $r=1$ and $s=2$, the sequence of the numbers $u_n={\rm Z}({\text U}_{n,0}^{(1,2)})$ is given by 
$$
2, 4, 12, 40, 136, 464, 1584, 5408, 18464, 63040, 215232,\dots 
$$ 
\cite[A056236]{oeis}. Similarly, the corresponding sequences for ${\text U}_{n,0}^{(1,3)}$, ${\text U}_{n,0}^{(1,4)}$, ${\text U}_{n,0}^{(2,3)}$ and ${\text U}_{n,0}^{(3,4)}$ are found in \cite[A228569,A228842,A094433,A074601]{oeis}, respectively.  

\begin{proof}[Proof of Theorem \ref{th:non-uniform}.] 
Let integers $n,r,s\ge 1$ and $m\ge 0$. 
Using Lemma \ref{gp}, we divide the graph ${\rm U}_{n,m}^{(r,s)}$ into two subgraphs. By repeating $s$ times to remove one edge from the $s$-hold bond part, we have 
\begin{multline}
u_n:={\text Z}({\rm U}_{n,m}^{(r,s)})\\
={\text Z}\bigl(D_{2 n}(m+1,\dots,m+1;r,s,r,\dots,r)\bigr)+s\cdot{\text Z}\bigl(D_{2 n-2}(m+1,\dots,m+1;s,r,s,\dots,s)\bigr)\,. 
\label{eq:g-2subg}
\end{multline}  
For the first subgraph, we set 
$$
\frac{f_{2 n}}{f_{2 n-1}}:=\underbrace{m+1+\frac{r}{m+1}{\atop+}\frac{s}{m+1}{\atop+}\frac{r}{m+1}{\atop+\dots+}\frac{r}{m+1}}_{2 n}\,,  
$$ 
so that $f_{2 n}={\text Z}\bigl(D_{2 n}(m+1,\dots,m+1;r,s,r,\dots,r)\bigr)$. Then the sequence $\{f_n\}_n$ satisfies the recurrence relation 
$$
f_n=(m+1)f_{n-1}+
\begin{cases}
r f_{n-2}&\text{if $n$ is even};\\
s f_{n-2}&\text{if $n$ is odd}
\end{cases}
$$ 
with $f_0=1$ and $f_1=m+1$.  Thus, the sequence $\{f_{2 n}\}_n$ satisfies the recurrence relation 
\begin{align*}  
f_{2 n}&=(m+1)f_{2 n-1}+r f_{2 n-2}\\
&=(m+1)\bigl((m+1)f{2 n-2}+s f_{2 n-3}\bigr)+r f_{2 n-2}\\
&=\bigl((m+1)^2+r\bigr)f_{2 n-2}+s(f_{2 n-2}-r f_{2 n-4})\\ 
&=\bigl((m+1)^2+r+s\bigr)f_{2 n-2}-r s f_{2 n-4}
\end{align*}
with $f_0=1$ and $f_2=(m+1)^2+r$.  

For the second subgraph, we set 
$$
\frac{g_{2 n}}{g_{2 n-1}}:=\underbrace{m+1+\frac{s}{m+1}{\atop+}\frac{r}{m+1}{\atop+}\frac{s}{m+1}{\atop+\dots+}\frac{s}{m+1}}_{2 n-2}\,,  
$$ 
so that $g_{2 n}={\text Z}\bigl(D_{2 n-2}(m+1,\dots,m+1;s,r,s,\dots,s)\bigr)$.  
Then the sequence $\{g_n\}_n$ satisfies the recurrence relation 
$$
g_n=(m+1)g_{n-1}+
\begin{cases}
s g_{n-2}&\text{if $n$ is even};\\
r g_{n-2}&\text{if $n$ is odd}
\end{cases}
$$ 
with $g_0=1/s$, $g_1=0$ and $g_2=1$.  Thus, similarly, the sequence $\{g_{2 n}\}_n$ satisfies the recurrence relation 
$$  
g_{2 n}=\bigl((m+1)^2+r+s\bigr)g_{2 n-2}-r s g_{2 n-4}
$$ 
with $g_0=1/s$, $g_2=1$ and $g_4=(m+1)^2+s$.  

Since both $f_{2 n}$ and $g_{2 n}$ satisfy the same recurrence relation, the sequence $\{u_n\}_n$ satisfy the same recurrence relation $u_n=M u_{n-1}-r s u_{n-2}$ with 
 $u_0=f_0+s g_0=2$ and $u_1=f_2+s g_2=(m+1)^2+r+s:=M$. 
Therefore,  
\begin{align*}  
\frac{u_n}{u_{n-1}}&=M-\dfrac{r s}{\frac{u_{n-1,m}}{u_{n-2,m}}}\\ 
&=M-\cfrac{r s}{M-\cfrac{r s}{\frac{u_{n-2,m}}{u_{n-3,m}}}}\\
&=\cdots\\
&=M-\cfrac{r s}{M-\cfrac{r s}{M-{\atop\ddots-\dfrac{r s}{\frac{u_{1}}{u_{0}}}}}}\\
&=M-\cfrac{r s}{M-\cfrac{r s}{M-{\atop\ddots-\dfrac{r s}{M-\dfrac{M}{2}}}}}\,.
\end{align*} 

In addition, since the roots of $x^2-M x+r s=0$ are given by 
$$
\alpha,\beta=\frac{M\pm\sqrt{M^2-4 r s}}{2}\,, 
$$ 
together with the initial values, we have $u_n=\alpha^n+\beta^n$.   
\end{proof}

\section{Radial graph and multidimensional continued fractions}  

Another typical graph is that of radial crystals represented by ice crystals. Here, we consider the fact that each of the radial shapes is needle-shaped, and show the relationship with multidimensional continued fractions. 

We repeatedly use the convenient method (\cite{Komatsu2020}) by using the branched continued fractions ${\rm D}_n(A_0,a_1,\dots,a_{n-1};b_1,\dots,b_{n-1})\bigr)$ with a positive rational number $A_0$. 

\begin{Lem}
For some $i$ ($i\ge 1$), let $A_i$ be a positive rational number which continued fraction is given by 
$$
A_i=\frac{P_i}{Q_i}=1+\cfrac{d_1}{c_1+{\atop\ddots{\atop\displaystyle +\dfrac{d_i}{c_i}}}}
$$ 
for positive integers $c_j$ and $d_j$ ($j\ge 1$), 
according to the similar recurrence relations (\ref{rec:pn}) and (\ref{rec:qn}).  Then, for positive integers $a_h$ ($h\ne i$) and $b_h$, 
the Hosoya index of above combined caterpillar-bond graph is equal to  
$$
{\rm Z}\bigl({\rm D}_n(a_0,\dots,a_{i-1},A_i,a_{i+1},\dots,a_{n-1};b_1,\dots,b_{n-1})\bigr)=p_{n-1}\,, 
$$ 
where a positive integer $a_i$ in (\ref{pnqn}) is replaced by $A_i$.  
\label{lem:branched} 
\end{Lem} 
 
In order to illustrate the general theory, consider the following radial graph in which each radial part is represented by the same caterpillar-bond graph $D_3(2,3,3;3,2)$. 

$$
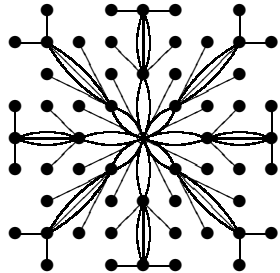
\xymatrix@=6pt{ 
&*=0{\bullet}\ar@{-}[d]&&*=0{\bullet}\ar@{-}[r]&*=0{\bullet}\ar@{-}[dd]\ar @/^0.2pc/ @{-}[dd]\ar @/_0.2pc/ @{-}[dd]&*=0{\bullet}\ar@{-}[l]&&*=0{\bullet}\ar@{-}[d]&\\
*=0{\bullet}\ar@{-}[r]&*=0{\bullet}\ar@{-}[rrdd]\ar @/^0.2pc/ @{-}[rrdd]\ar @/_0.2pc/ @{-}[rrdd]&*=0{\bullet}\ar@{-}[rdd]&*=0{\bullet}\ar@{-}[rd]&&*=0{\bullet}\ar@{-}[ld]&*=0{\bullet}\ar@{-}[ldd]&*=0{\bullet}\ar@{-}[lldd]\ar @/^0.2pc/ @{-}[lldd]\ar @/_0.2pc/ @{-}[lldd]&*=0{\bullet}\ar@{-}[l]\\
&*=0{\bullet}\ar@{-}[rrd]&&*=0{\bullet}\ar@{-}[rdd]&*=0{\bullet}\ar @/^0.2pc/ @{-}[dd]\ar @/_0.2pc/ @{-}[dd]&*=0{\bullet}\ar@{-}[ldd]&&*=0{\bullet}\ar@{-}[lld]&\\
*=0{\bullet}\ar@{-}[d]&*=0{\bullet}\ar@{-}[rd]&*=0{\bullet}\ar@{-}[rrd]&*=0{\bullet}\ar @/^0.2pc/ @{-}[rd]\ar @/_0.2pc/ @{-}[rd]&&*=0{\bullet}\ar @/^0.2pc/ @{-}[ld]\ar @/_0.2pc/ @{-}[ld]&*=0{\bullet}\ar@{-}[lld]&*=0{\bullet}\ar@{-}[ld]&*=0{\bullet}\ar@{-}[d]\\
*=0{\bullet}\ar@{-}[rr]\ar @/^0.2pc/ @{-}[rr]\ar @/_0.2pc/ @{-}[rr]&&*=0{\bullet}\ar @/^0.2pc/ @{-}[rr]\ar @/_0.2pc/ @{-}[rr]&&*=0{\circ}\ar @/^0.2pc/ @{-}[dd]\ar @/_0.2pc/ @{-}[dd]\ar @/^0.2pc/ @{-}[rr]\ar @/_0.2pc/ @{-}[rr]&&*=0{\bullet}\ar@{-}[rr]\ar @/^0.2pc/ @{-}[rr]\ar @/_0.2pc/ @{-}[rr]&&*=0{\bullet}\\
*=0{\bullet}\ar@{-}[u]&*=0{\bullet}\ar@{-}[ru]&*=0{\bullet}\ar@{-}[rru]&*=0{\bullet}\ar @/^0.2pc/ @{-}[ru]\ar @/_0.2pc/ @{-}[ru]&&*=0{\bullet}\ar @/^0.2pc/ @{-}[lu]\ar @/_0.2pc/ @{-}[lu]&*=0{\bullet}\ar@{-}[llu]&*=0{\bullet}\ar@{-}[lu]&*=0{\bullet}\ar@{-}[u]\\
&*=0{\bullet}\ar@{-}[rru]&&*=0{\bullet}\ar@{-}[ruu]&*=0{\bullet}\ar@{-}[dd]\ar @/^0.2pc/ @{-}[dd]\ar @/_0.2pc/ @{-}[dd]&*=0{\bullet}\ar@{-}[luu]&&*=0{\bullet}\ar@{-}[llu]&\\
*=0{\bullet}\ar@{-}[r]&*=0{\bullet}\ar@{-}[rruu]\ar @/^0.2pc/ @{-}[rruu]\ar @/_0.2pc/ @{-}[rruu]&*=0{\bullet}\ar@{-}[ruu]&*=0{\bullet}\ar@{-}[ru]&&*=0{\bullet}\ar@{-}[lu]&*=0{\bullet}\ar@{-}[luu]&*=0{\bullet}\ar@{-}[lluu]\ar @/^0.2pc/ @{-}[lluu]\ar @/_0.2pc/ @{-}[lluu]&*=0{\bullet}\ar@{-}[l]\\
&*=0{\bullet}\ar@{-}[u]&&*=0{\bullet}\ar@{-}[r]&*=0{\bullet}&*=0{\bullet}\ar@{-}[l]&&*=0{\bullet}\ar@{-}[u]&
} 
$$

If there is only one radial part, by 
$$
2+\cfrac{2}{3+\cfrac{3}{3}}=\frac{30}{12}\,, \qquad\qquad 
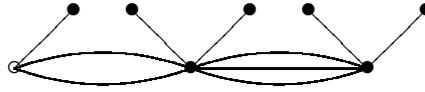
\xymatrix@=16pt{ 
&*=0{\bullet}&*=0{\bullet}&&*=0{\bullet}&*=0{\bullet}&&*=0{\bullet}\\
*=0{\circ}\ar @{-} [ru]&&&*=0{\bullet}\ar @{-} [lu]\ar @{-} [ru]\ar @/^/ @{-}[rrr]\ar @{-} [rrr] \ar @{-} [rrr]\ar @/_/ @{-}[rrr]\ar @/^/ @{-}[lll]\ar @/_/ @{-}[lll]&&&*=0{\bullet}\ar @{-} [lu]\ar @{-} [ru]&
}
$$ 
the topological index is given by $30$.  
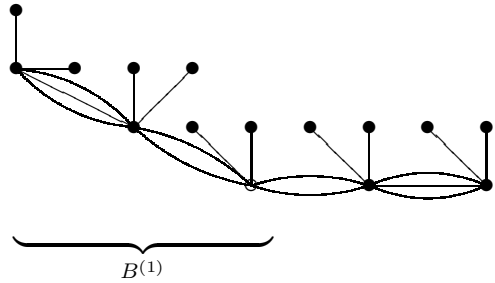
\begin{align*}
&\xymatrix@=16pt{ 
*=0{\bullet}\ar@{-}[d]&&&&&&&&\\
*=0{\bullet}\ar@{-}[rrd]\ar @/^0.4pc/ @{-}[rrd]\ar @/_0.4pc/ @{-}[rrd]&*=0{\bullet}\ar@{-}[l]&*=0{\bullet}\ar@{-}[d]&*=0{\bullet}\ar@{-}[ld]&&&&&\\
&&*=0{\bullet}\ar @/^0.3pc/ @{-}[rrd]\ar @/_0.3pc/ @{-}[rrd]&*=0{\bullet}\ar@{-}[rd]&*=0{\bullet}\ar@{-}[d]&*=0{\bullet}\ar@{-}[rd]&*=0{\bullet}\ar@{-}[d]&*=0{\bullet}\ar@{-}[rd]&*=0{\bullet}\ar@{-}[d]\\
&&&&*=0{\circ}\ar @/^0.3pc/ @{-}[rr]\ar @/_0.3pc/ @{-}[rr]&&*=0{\bullet}\ar@{-}[rr]\ar @/^0.4pc/ @{-}[rr]\ar @/_0.4pc/ @{-}[rr]&&*=0{\bullet}
}\\
&\underbrace{\phantom{AAAAAAAAAAA}}_{B^{(1)}}
\end{align*} 
If there are two radial parts, by applying Lemma \ref{lem:branched} to  
$$
B^{(1)}+\cfrac{2}{3+\cfrac{3}{3}},\quad\hbox{where}\quad B^{(1)}=3+\cfrac{2}{3+\cfrac{3}{3}}\,, 
$$ 
we have 
\begin{align*}
3+\cfrac{2}{3+\cfrac{3}{3}}+\cfrac{2}{3+\cfrac{3}{3}}&=3+\frac{6}{12}+\frac{6}{12}\\
&=\frac{576}{144}\,.   
\end{align*}  
Hence, the topological index is given by $576$.  
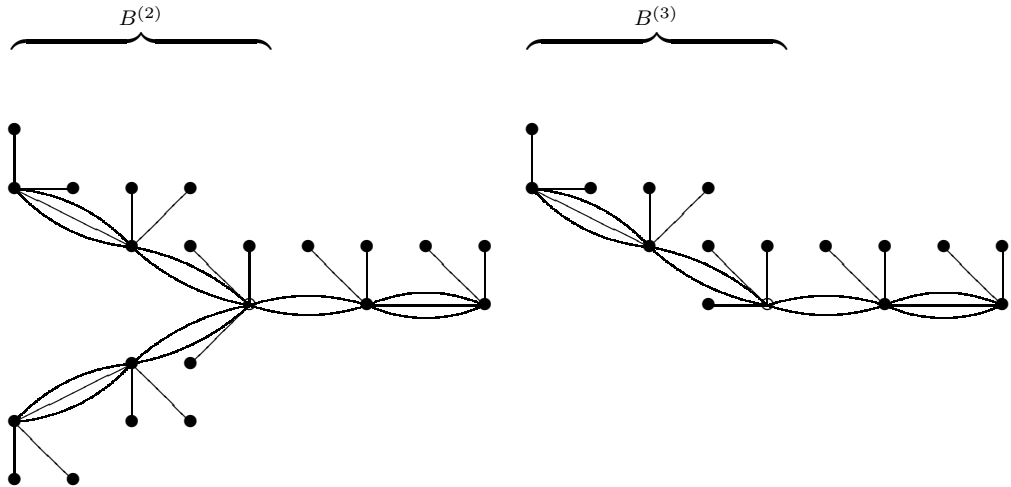
\begin{align*}
&\overbrace{\phantom{AAAAAAAAAAA}}^{B^{(2)}}\qquad\qquad\qquad\qquad\overbrace{\phantom{AAAAAAAAAAA}}^{B^{(3)}}\\
&\xymatrix@=16pt{ 
*=0{\bullet}\ar@{-}[d]&&&&&&&&\\
*=0{\bullet}\ar@{-}[rrd]\ar @/^0.4pc/ @{-}[rrd]\ar @/_0.4pc/ @{-}[rrd]&*=0{\bullet}\ar@{-}[l]&*=0{\bullet}\ar@{-}[d]&*=0{\bullet}\ar@{-}[ld]&&&&&\\
&&*=0{\bullet}\ar @/^0.3pc/ @{-}[rrd]\ar @/_0.3pc/ @{-}[rrd]&*=0{\bullet}\ar@{-}[rd]&*=0{\bullet}\ar@{-}[d]&*=0{\bullet}\ar@{-}[rd]&*=0{\bullet}\ar@{-}[d]&*=0{\bullet}\ar@{-}[rd]&*=0{\bullet}\ar@{-}[d]\\
&&&&*=0{\circ}\ar @/^0.3pc/ @{-}[rr]\ar @/_0.3pc/ @{-}[rr]&&*=0{\bullet}\ar@{-}[rr]\ar @/^0.4pc/ @{-}[rr]\ar @/_0.4pc/ @{-}[rr]&&*=0{\bullet}\\
&&*=0{\bullet}\ar @/^0.3pc/ @{-}[rru]\ar @/_0.3pc/ @{-}[rru]&*=0{\bullet}\ar@{-}[ru]&&&&&\\
*=0{\bullet}\ar@{-}[rru]\ar @/^0.4pc/ @{-}[rru]\ar @/_0.4pc/ @{-}[rru]&&*=0{\bullet}\ar@{-}[u]&*=0{\bullet}\ar@{-}[lu]&&&&&\\
*=0{\bullet}\ar@{-}[u]&*=0{\bullet}\ar@{-}[lu]&&&&&&&
}\quad 
\xymatrix@=16pt{ 
*=0{\bullet}\ar@{-}[d]&&&&&&&&\\
*=0{\bullet}\ar@{-}[rrd]\ar @/^0.4pc/ @{-}[rrd]\ar @/_0.4pc/ @{-}[rrd]&*=0{\bullet}\ar@{-}[l]&*=0{\bullet}\ar@{-}[d]&*=0{\bullet}\ar@{-}[ld]&&&&&\\
&&*=0{\bullet}\ar @/^0.3pc/ @{-}[rrd]\ar @/_0.3pc/ @{-}[rrd]&*=0{\bullet}\ar@{-}[rd]&*=0{\bullet}\ar@{-}[d]&*=0{\bullet}\ar@{-}[rd]&*=0{\bullet}\ar@{-}[d]&*=0{\bullet}\ar@{-}[rd]&*=0{\bullet}\ar@{-}[d]\\
&&&*=0{\bullet}\ar@{-}[r]&*=0{\circ}\ar @/^0.3pc/ @{-}[rr]\ar @/_0.3pc/ @{-}[rr]&&*=0{\bullet}\ar@{-}[rr]\ar @/^0.4pc/ @{-}[rr]\ar @/_0.4pc/ @{-}[rr]&&*=0{\bullet}
}
\end{align*} 
If there are three radial parts, by applying Lemma \ref{lem:branched} two times to 
\begin{align*}
B^{(2)}+\cfrac{2}{3+\cfrac{3}{3}},\quad\hbox{where}\quad &B^{(2)}=B^{(3)}+\cfrac{2}{3+\cfrac{3}{3}}\\
\hbox{and}\quad &B^{(3)}=4+\cfrac{2}{3+\cfrac{3}{3}}  
\end{align*} 
we have 
\begin{align*}
4+\cfrac{2}{3+\cfrac{3}{3}}+\cfrac{2}{3+\cfrac{3}{3}}+\cfrac{2}{3+\cfrac{3}{3}}&=3+\frac{6}{12}+\frac{6}{12}+\frac{6}{12}\\
&=\frac{9504}{1728}\,.   
\end{align*}  
Hence, the topological index is given by $9504$.  In general, if there are $m$ radial parts, the corresponding continued fraction becomes 
$$
m+1+\underbrace{\cfrac{2}{3+\cfrac{3}{3}}+\cdots+\cfrac{2}{3+\cfrac{3}{3}}}_m=\frac{(18 m+12)\cdot 12^{m-1}}{12^m}\,. 
$$ 
This type of continued fractions is called the {\it multidimensional continued fraction}. Then, the topological index is given by $(18 m+12)\cdot 12^{m-1}$.  
In the case of the first figure the Hosoya index is given by $(18\cdot 8+12)\cdot 12^7=5589762048$ because this radial graph corresponds with the $8$-dimensional continued fraction  
$$
9+\underbrace{\cfrac{2}{3+\cfrac{3}{3}}+\cdots+\cfrac{2}{3+\cfrac{3}{3}}}_8=\frac{(18\cdot 8+12)\cdot 12^7}{12^8}\,. 
$$

In general, consider a $m$-dimensional radial graph in which each radial part $\mathfrak R$ is expressed as follows.  
$$ 
a_0+\frac{b_1}{A_1}{\atop+}\frac{b_2}{A_2}{\atop+\dots+}\frac{b_n}{A_n}\qquad\qquad 
\xymatrix@=16pt{ 
*=0{\bullet}&*=0{\bullet}&*=0{\bullet}&&*=0{\phantom{A}}&&*=0{\phantom{A}}&&&&*=0{\phantom{A}}&\\
&*=0{\circ}\ar@{-}[lu]^{a_0-1}\ar@{-}[u]\ar@{-}[ru]\ar@/^/@{-}[rrr]\ar@/^0.2pc/@{-}[rrr]\ar@/_0.2pc/@{-} [rrr]  
\ar@/_/@{-}[rrr]_{b_1}&&&*=0{\bullet}\ar@{.}[u]\ar@/^0.2pc/@{.}[u]_{\mathfrak A_1}\ar@/_0.2pc/@{.}[u]\ar@/^/@{-}[rr]\ar@/^0.2pc/@{-}[rr]\ar@/_0.2pc/@{-} [rr] 
\ar@/_/@{-}[rr]_{b_2}&&*=0{\bullet}\ar@{.}[u]\ar@/^0.2pc/@{.}[u]\ar@/_0.2pc/@{.}[u]_{\mathfrak A_2}\ar @/^/@{.}[rr]\ar@/^0.2pc/@{.}[rr]\ar@/_0.2pc/@{.}[rr]  
\ar@/_/@{.}[rr]&&*=0{\bullet}&&*=0{\bullet}\ar@{.}[u]\ar@/^0.2pc/@{.}[u]\ar@/_0.2pc/@{.}[u]_{\mathfrak A_n}\ar@/^/@{-}[ll]^{b_{n}}\ar@/^0.2pc/@{-}[ll]\ar@/_0.2pc/@{-}[ll]\ar@/_/@{-}[ll]&
}
$$  
Here, $a_0$, $b_1,\dots,b_n$ are positive integers. $A_1,\dots,A_n$ are positive rational numbers. If $A_i$ ($i=1,2,\dots,n$) is a positive integer, the corresponding subgraph $\mathfrak A_i$ denotes $A_i-1$ branches. If $A_i$ is a positive fraction, the corresponding subgraph $\mathfrak A_i$ denotes another caterpillar-bond graph or caterpillar-bond graph including still more caterpillar-bond graphs. Then $A_i$ can be expressed as a continued fraction or a branched continued fraction.     

By applying Lemma \ref{lem:branched} repeatedly, we have the following.  

\begin{theorem}
The Hosoya index of the $m$-fold radial graph, in which each radial part $\mathfrak R$ is  connected only by one vertex marked with white circle,  
is given by the numerator of the convergent of the multidimensional continued fraction 
$$
m(a_0-1)+1+\underbrace{\frac{b_1}{A_1}{\atop+}\frac{b_2}{A_2}{\atop+\dots+}\frac{b_n}{A_n}+\cdots+\frac{b_1}{A_1}{\atop+}\frac{b_2}{A_2}{\atop+\dots+}\frac{b_n}{A_n}}_{m}\,. 
$$ 
\label{th:radial} 
\end{theorem} 

When $m=2$, by setting 
$b:=b_1=b_2=\dots$ and 
$$
A_1=A_2=\dots=m'+\cfrac{b}{m'+\cfrac{b}{m'+\cfrac{b}{\ddots}+\cfrac{b}{\ddots}}+\cfrac{b}{m'+\cfrac{b}{\ddots}+\cfrac{b}{\ddots}}}\,,  
$$ 
a periodic $2$-dimensional continued fraction (\cite{BK}) with period $1$ can be yielded as 
\begin{gather*}m'+\dfrac{b}{m'+\dfrac{b}{m'+\dfrac{b}{m'+{}_{\ddots}}+\dfrac{b}{m'+{}_{\ddots}}}+
    \dfrac{b}{m'+\dfrac{b}{m'+{}_{\ddots}}+\dfrac{b}{m'+{}_{\ddots}}}}\\+\dfrac{b}{m'+\dfrac{b}{m'+\dfrac{b}{m'+{}_{\ddots}}+\dfrac{b}{m'+{}_{\ddots}}}+
    \dfrac{b}{m'+\dfrac{b}{m'+{}_{\ddots}}+\dfrac{b}{m'+{}_{\ddots}}}},\end{gather*}
where $m'=2(a_0-1)+1$. 
  
For example, for $a_0=b=2$, the third convergent of the periodic $2$-dimensional continued fraction with period $1$ is given by 
$$
3+\cfrac{2}{3+\cfrac{2}{3+\cfrac{2}{3}+\cfrac{2}{3}}+\cfrac{2}{3+\cfrac{2}{3}+\cfrac{2}{3}}}+\cfrac{2}{3+\cfrac{2}{3+\cfrac{2}{3}+\cfrac{2}{3}}+\cfrac{2}{3+\cfrac{2}{3}+\cfrac{2}{3}}}\left(=\frac{143118495}{35605089}\right)\,, 
$$ 
and the corresponding graph is given by the following.  
$$
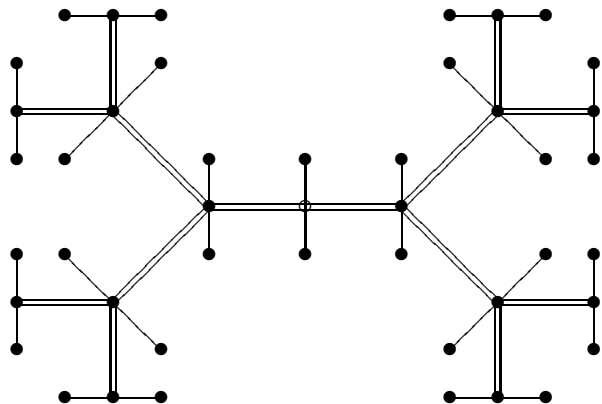
\xymatrix@=12pt{ 
&*=0{\bullet}\ar@{-}[r]&*=0{\bullet}\ar@{=}[dd]&*=0{\bullet}\ar@{-}[l]&&&&&&*=0{\bullet}\ar@{-}[r]&*=0{\bullet}\ar@{=}[dd]&*=0{\bullet}\ar@{-}[l]&\\
*=0{\bullet}\ar@{-}[d]&&&*=0{\bullet}\ar@{-}[ld]&&&&&&*=0{\bullet}\ar@{-}[rd]&&&*=0{\bullet}\ar@{-}[d]\\
*=0{\bullet}\ar@{=}[rr]&&*=0{\bullet}\ar@{=}[rrdd]&&&&&&&&*=0{\bullet}\ar@{=}[lldd]&&*=0{\bullet}\ar@{=}[ll]\\
*=0{\bullet}\ar@{-}[u]&*=0{\bullet}\ar@{-}[ru]&&&*=0{\bullet}\ar@{-}[d]&&*=0{\bullet}\ar@{-}[d]&&*=0{\bullet}\ar@{-}[d]&&&*=0{\bullet}\ar@{-}[lu]&*=0{\bullet}\ar@{-}[u]\\
&&&&*=0{\bullet}\ar@{=}[rr]\ar@{=}[lldd]&&*=0{\circ}&&*=0{\bullet}\ar@{=}[ll]\ar@{=}[rrdd]&&&&\\
*=0{\bullet}\ar@{-}[d]&*=0{\bullet}\ar@{-}[rd]&&&*=0{\bullet}\ar@{-}[u]&&*=0{\bullet}\ar@{-}[u]&&*=0{\bullet}\ar@{-}[u]&&&*=0{\bullet}\ar@{-}[ld]&*=0{\bullet}\ar@{-}[d]\\
*=0{\bullet}\ar@{=}[rr]&&*=0{\bullet}\ar@{=}[dd]&&&&&&&&*=0{\bullet}\ar@{=}[dd]&&*=0{\bullet}\ar@{=}[ll]\\
*=0{\bullet}\ar@{-}[u]&&&*=0{\bullet}\ar@{-}[lu]&&&&&&*=0{\bullet}\ar@{-}[ru]&&&*=0{\bullet}\ar@{-}[u]\\
&*=0{\bullet}\ar@{-}[r]&*=0{\bullet}&*=0{\bullet}\ar@{-}[l]&&&&&&*=0{\bullet}\ar@{-}[r]&*=0{\bullet}&*=0{\bullet}\ar@{-}[l]&
}
$$ 
Therefore, the Hosoya index of this graph is given by $143118495$.

\section{Final comments}  

The graph ${\rm U}_{n,m}^{(r,s)}$ has a corresponding negative continued fraction. However, when this graph attachs other graph or the number of branches of each vertex is different, we still have no idea how to calculate its Hosoya index efficiently, in particular, by using continued fractions. For example, no suitable continued fraction has been found for the following combined graph. The right side is known as the molecular formula of naphthalene ${\rm C}_{10}{\rm H}_{8}$.  
$$
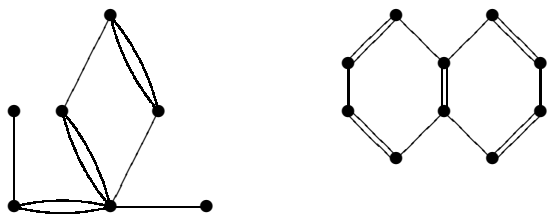
\xymatrix@=12pt{ 
&&*=0{\bullet}\ar@{-}[ldd]\ar@/^0.2pc/@{-}[rdd]\ar@/_0.2pc/@{-} [rdd]&&\\
&&&&\\
*=0{\bullet}\ar@{-}[dd]&*=0{\bullet}\ar@/^0.2pc/@{-}[rdd]\ar@/_0.2pc/@{-} [rdd]&&*=0{\bullet}\ar@{-}[ldd]&\\
&&&&\\
*=0{\bullet}\ar@/^0.2pc/@{-}[rr]\ar@/_0.2pc/@{-} [rr]&&*=0{\bullet}&&*=0{\bullet}\ar@{-}[ll]
} 
\qquad\qquad 
\xymatrix@=12pt{ 
&*=0{\bullet}\ar@{=}[ld]\ar@{-}[rd]&&*=0{\bullet}\ar@{-}[ld]\ar@{=}[rd]&\\
*=0{\bullet}\ar@{-}[d]&&*=0{\bullet}\ar@{=}[d]&&*=0{\bullet}\ar@{-}[d]\\
*=0{\bullet}\ar@{=}[rd]&&*=0{\bullet}\ar@{-}[ld]\ar@{-}[rd]&&*=0{\bullet}\ar@{=}[ld]\\
&*=0{\bullet}&&*=0{\bullet}&
} 
$$

\section{Acknowledgment} 

The author thanks the anonymous referee for detailed comments and suggestions, which improved the quality of the paper.

\end{document}